\documentclass[reqno]{amsart}
\usepackage{paralist}
\usepackage{amsthm, amsfonts, amssymb, amsmath, latexsym, array,color,amscd}
\usepackage[margin=1.2in]{geometry}
\usepackage{courier}
\usepackage{dirtytalk}
\usepackage[utf8]{inputenc}
\usepackage[T1]{fontenc}
\usepackage[english]{babel}
\usepackage{tgpagella}
\usepackage{caption} 
\usepackage[colorlinks=true, pdfstartview=FitV, linkcolor=blue, citecolor=blue, urlcolor=blue]{hyperref}
\usepackage[capitalise, noabbrev]{cleveref}
\usepackage{csquotes}
\usepackage{fancyhdr}
\usepackage{multirow}
\usepackage{comment}
\usepackage{soul}
\usepackage{mathtools}
\usepackage{amsmath}
\usepackage[all]{xy}
\usepackage{pstricks,graphicx}
\usepackage{stmaryrd}
\usepackage{mdwlist,mathrsfs}
\usepackage{pgf,tikz}
\usepackage{tikz-cd,rotating}
\usepackage{pdflscape}
\usepackage{array}

\usetikzlibrary{arrows}
\usepackage{todonotes}
\usepackage{color}
\usepackage{lipsum}
\allowdisplaybreaks

\newtheorem{theorem}{Theorem}[section]
\newtheorem*{theorem*}{Theorem}
\newtheorem{lemma}[theorem]{Lemma}
\newtheorem{question}[theorem]{Question}
\newtheorem{proposition}[theorem]{Proposition}
\newtheorem{corollary}[theorem]{Corollary}

\newtheorem*{proposition*}{Proposition}
\newtheorem{problem}[theorem]{Problem}
\newtheorem*{problem*}{Open Problems}
\theoremstyle{definition}
\newtheorem{definition}[theorem]{Definition}
\newtheorem{example}[theorem]{Example}
\newtheorem{remark}[theorem]{Remark}

\numberwithin{equation}{section}

\DeclareMathOperator{\Ann}{Ann}

\DeclareMathOperator{\Ext}{Ext}

\DeclareMathOperator{\Hess}{Hess}
\DeclareMathOperator{\Hom}{Hom}

\DeclareMathOperator{\reg}{reg}
\DeclareMathOperator{\lt}{lt}
\DeclareMathOperator{\lcm}{lcm}

\newcommand {\PP}{\mathbb{P}}

\newcommand{\kk}{\ensuremath{K}}

\newcommand{\HF}{\ensuremath{\mathrm{HF}}}

\newcommand{\qand}{\quad \mbox{and} \quad}

\newcommand{\qwhere}{\quad \mbox{where} \quad}
\newcommand{\qwith}{\quad \mbox{with} \quad}
\newcommand{\qfor}{\quad \mbox{for} \quad}
\newcommand{\qforall}{\quad \mbox{for all} \quad}

\newcommand{\m}{\mathbf{m}}

\renewcommand{\char}{\rm char}

\usepackage{enumitem,tabstackengine}
\TABstackMath

\title[AG algebras with binomial Macaulay dual generator]{Artinian Gorenstein algebras with binomial Macaulay dual generator}

\author[]{Nasrin Altafi}
\address{Nasrin Altafi: Department of Mathematics, KTH Royal Institute of Technology, S-100 44 Stockholm, Sweden and Department of Mathematics, Queen's University, 505 Jeffery Hall, University Avenue, Queen's University, Kingston, Ontario, Canada K7L 3N6}
\email{nar3@queensu.ca}

 \author[]{Rodica Dinu}
 \address{Rodica Dinu: University of Konstanz, Fachbereich Mathematik und Statistik, Fach D 197 D-78457 Konstanz, Germany, and Simion Stoilow Institute of Mathematics of the Romanian Academy, Calea Grivitei 21, 010702, Bucharest, Romania}
  \email{rodica.dinu@uni-konstanz.de}

\author[]{Sara Faridi}
\address{Sara Faridi: Department of Mathematics and Statistics,
Dalhousie University,
6297 Castine Way,
PO BOX 15000.
Halifax, NS.
Canada B3H 4R2}
\email{faridi@dal.ca}

\author[]{Shreedevi K. Masuti}
\address{Shreedevi K. Masuti: Department of Mathematics, Indian Institute of Technology Dharwad, Permanent Campus, Chikkamalligwad,  Dharwad - 580011, Karnataka, India}
\email{shreedevi@iitdh.ac.in}

  \author[]{Rosa M.\ Mir\'o-Roig}
  \address{Rosa Maria Mir\'o-Roig: Facultat de
  Matem\`atiques i Inform\`atica, Universitat de Barcelona, Gran Via des les
  Corts Catalanes 585, 08007 Barcelona, Spain} \email{miro@ub.edu,  ORCID 0000-0003-1375-6547}

\author[]{Alexandra Seceleanu}
\address{Alexandra Seceleanu: Department of
  Mathematics, University of Nebraska-Lincoln, 203 Avery Hall, Lincoln, NE 68588, USA}
\email{aseceleanu@unl.edu}

\author[]{Nelly Villamizar}
\address{Nelly Villamizar: Department of Mathematics, Swansea University, Fabian Way, SA1 8EN, Swansea, UK
}
\email{n.y.villamizar@swansea.ac.uk}

\thanks{\hspace{-12pt}Altafi was supported by Swedish Research Council grant VR2021-00472, Dinu was supported by the Alexander von Humboldt Foundation and DFG grant nr 467575307, Faridi was supported by an NSERC Discovery grant~2023-05929, Masuti is supported by CRG grant CRG/2022/007572 and MATRICS grant MTR/2022/000816 funded by ANRF,  Govt. of India, Mir\'o-Roig was partially supported by the grant PID2020-113674GB-I00,
Seceleanu was partially supported by NSF DMS–2101225 and DMS–2401482, Villamizar was partially supported by the EPSRC New Investigator Award EP/V012835/1.}

\begin{document}

\begin{abstract} 
This paper initiates a systematic study for key properties of Artinian Gorenstein \(K\)-algebras having binomial Macaulay dual generators. In codimension 3, we demonstrate that all such algebras satisfy the strong Lefschetz property, can be constructed as a doubling of an appropriate 0-dimensional scheme in \(\mathbb{P}^2\), and we provide an explicit  characterization of when they form a complete intersection. For arbitrary codimension, we establish sufficient conditions under which the weak Lefschetz property holds and show that these conditions are optimal.
\end{abstract}

\maketitle

\date{}

\setcounter{tocdepth}{1}
\tableofcontents

\section{Introduction}
The study of graded Artinian Gorenstein \(K\)-algebras holds a significant place in commutative algebra  due to their rich structure and numerous applications. Central to the analysis of these algebras is the concept of a Macaulay dual generator, which, via Macaulay-Matlis duality, connects every finitely generated Artinian Gorenstein graded \(K\)-algebra \(A_F\) to a single homogeneous polynomial \(F\). This duality not only provides a parametrization of these algebras but also transforms algebraic properties into questions about the polynomial \(F\), streamlining their exploration.

When the Macaulay dual generator $F$ is a monomial, the associated Artinian  Gorenstein $K$-algebra $A_F$ is a monomial complete intersection, by which we mean a quotient of a polynomial ring by an ideal generated by pure powers of each of the variables. The case where \(F\) is monomial has been extensively studied, yielding insights into complete intersections, Koszul resolutions, and Lefschetz properties.
In more detail, this class of rings satisfies exceptional properties among which we single out a few:
\begin{inparaenum}
    \item The minimal generators of the ideal defining a monomial complete intersection form a Gr\"obner basis;
    \item The minimal free resolutions for (monomial) complete intersections are given by the Koszul complex;
    \item Monomial complete intersections satisfy the strong Lefschetz property in characteristic zero \cite{stanley2,watanabe};
    \item (Monomial) complete intersections can be obtained by a doubling construction; see \cref{doubling_const}.
\end{inparaenum}

The weak Lefschetz property (WLP) and strong Lefschetz property (SLP) have become fundamental concepts in commutative algebra on the basis of their deep connections to algebraic geometry and topology, and applications in combinatorics \cite{LefscetzBook}. These properties describe the action of a linear form on the graded components of Artinian algebras, often signaling a balance in their growth and decay patterns. For instance, algebras satisfying Lefschetz properties  exhibit unimodal Hilbert functions. The Lefschetz properties are related to syzygies and free resolutions, and play a crucial role in understanding the geometry of projective varieties \cite{MNtour}. Despite significant progress, several important question in this area  remain open. A major unresolved issue is to determine whether all complete intersection rings of characteristic zero satisfy the WLP or SLP, and similarly for codimension three Artinian Gorenstein rings \cite{HMNW}.

The doubling construction  is a technique used to build a Gorenstein ring \( R/I \) from a Cohen-Macaulay ring \( R/J \). The construction is facilitated through a short exact sequence linking \( R/J \), its canonical module, and \( R/I \). This process effectively "doubles" the minimal free resolution of \( R/J \), combining it with its dual, resulting in a resolution for \( R/I \) that reflects this doubled structure. The doubling construction is particularly significant in Gorenstein liaison theory, as it provides a method to produce Gorenstein ideals as divisors on arithmetically Cohen-Macaulay schemes \cite{KMMNP}. This construction also offers unresolved questions in the classification and study of Gorenstein rings, primarily the question of classifying Gorenstein rings that are obtained by doubling \cite{KKRSSY}.

To summarize, beyond the case of monomial complete intersections, many problems regarding graded Artinian Gorenstein (AG) rings remain unresolved. 
 Some of the major open problems in this area, and relevant to this paper, are the following.

\begin{problem*}\label{Q1}
\phantom{a} \hfill
\begin{enumerate}[itemsep=0.1cm,topsep=0.1cm]
\item\label{o1} Characterize the Macaulay dual generators of complete intersection algebras \cite{HWW}.
\item\label{o2} Determine the  AG algebras which can be obtained by doubling \cite{KKRSSY}.
\item\label{o3}     Determine whether every AG algebra of codimension three, over a field of characteristic zero, satisfies the strong (or the weak) Lefschetz property \cite{HMNW}.
\item\label{o4}     Determine which AG algebras (of any codimension),  satisfy the strong (or the weak) Lefschetz property.
\end{enumerate}
\end{problem*}

Using the list above as a guide, in this paper we  study properties of AG algebras whose Macaulay dual generator is a binomial, that is, a $K$-linear combination of two monomials. Determining the structural and Lefschetz properties of Artinian Gorenstein \(K\)-algebras when \(F\) is a binomial remains a challenge and has been so far an unexplored topic. The binomial case bridges the gap between monomials and more general forms, providing a fertile ground for theoretical discovery. Therefore our results can be conceived as a case study to guide future research towards solving the above-mentioned problems.

For a binomial dual generator, we solve the first three problems listed above for algebras of codimension $3$,  and we provide sufficient conditions that answer Open Problem~\eqref{o4} in an optimal way in codimension $\ge 4$. Our results can be summarized as follows.

\begin{theorem*}[{\bf Main Results}]\label{thm: main}
Let $A_F$ be an Artinian Gorenstein $K$-algebra of codimension $c$ with binomial Macaulay dual generator $$F= m_1- m_2$$ where $m_1$ and $m_2$ are monomials of degree $d$. Then 
\begin{enumerate}[leftmargin=3em,itemsep=0.1cm,topsep=0.1cm]
\item if $c\ge 3$ and $\char(K)=0$ and $\deg(\gcd(m_1,m_2)) < \lfloor \frac{d-1}{2} \rfloor$, then $A_F$ satisfies the Weak  Lefschetz property (\cref{main1}); 
\item if $c=3$, we give explicit formulas for the minimal generators of $\Ann(F)$ (\cref{prop: gens family4}), providing  a criterion for when $A_F$ is a complete intersection ring (\cref{c:ci});
\item if $c=3$, the minimal generators of $\Ann(F)$ form a Gr\"obner basis with respect to a suitable monomial order (\cref{prop: GB});
\item if $c=3$, we determine the graded Betti numbers for $A_F$ (\cref{prop: gens family4});
\item if $c=3$ and $\char(K)=0$, then $A_F$ satisfies the strong Lefschetz property (\cref{thm: family 6});
\item if $c=3$, then $A_F$ is the doubling of a $0$-dimensional subscheme in $\PP^2$ (\cref{doubling2});
\item if $4 \leq c \leq  10$,  we give a specific binomial $F$ for which   $A_F$ fails the weak Lefschetz property  (\cref{e:failure}).
\end{enumerate}
Moreover, if $K$ is an algebraically closed field, then all assertions above hold more generally for $F=a \cdot m_1-b\cdot m_2$, where  $a,b$  are any  nonzero constants. 
\end{theorem*}
The assertions listed in the above theorem are proved as independent results throughout the paper, as referenced above. A reduction of  $F=a \cdot m_1-b\cdot m_2$ to $F=m_1-m_2$ under the assumption that $K$ contains all roots of unity is performed in \cref{lem: no constants}.

Our paper is structured as follows: In section 2, we gather the necessary background on Macaulay duality; connected sums; Lefschetz properties, and the doubling construction.
Section 3 contains our main result about AG algebras of arbitrary codimension  with  binomial Macaulay dual generator (see \cref{main1})  and, in sections 4, 5, and 6 we restrict our attention to the codimension 3 case. Section 4 focuses on the explicit description of the relations of an AG codimension 3 algebra $A_F$ with $F$ a binomial (see \cref{prop: gens family4,prop: GB}). In section 5 we prove that these algebras have the strong Lefschetz property (\cref{thm: family 6}) and in section 6 we return to structural and homological aspects establishing that such algebras can be obtain by means of  doubling (\cref{doubling2}).

\subsubsection*{Acknowledgements} This project started at the  meeting ``Women in Commutative Algebra II (WICA II)``  at CIRM Trento, Italy, October 16-20, 2023.  The authors would like to thank the CIRM and the organizers for the invitation and financial support through NSF DMS--2324929.
\section{Background}

Throughout this paper, $K$ will be a field. 
Given a standard graded Artinian $K$-algebra $A=R/I$, where $R=K[x_1,\dots,x_n]$ and $I$ is a homogeneous ideal of $R$,
we denote the Hilbert function of $A$ by $\HF_A\colon \mathbb{Z} \longrightarrow \mathbb{Z}$, with $\HF_A(j)=\dim _K[A]_j=\dim _K[R/I]_j$, and the Hilbert series of $A$ by $H_A(t)=\sum _i \HF_A(i)t^{i}$. 
Since $A$ is Artinian, its Hilbert function is
encoded in its \emph{$h$-vector} $h=(h_0,\dots ,h_d)$, where $h_i=\HF_A(i)>0$ and $d$ is the largest index with this property. The integer $d$ is called the \emph{socle degree of} $A$ or the regularity of $A$, and it is denoted by $\reg(A)$. We refer to the number $n=\dim(R)$, which is equal to the height or codimension of $I$, also as the {\em codimension} of $A$.

 \subsection{Gorenstein algebras and Macaulay duality}\label{s: Macaulay duality}

A graded Artinian $K$-algebra $A$  with socle degree $d$ is said to be  {\em Gorenstein} if its
socle $(0:\m_A)$ is a one dimensional $K$-vector space whose elements have degree $d$ where $\m_A=(x_1,\ldots,x_n)A$.

Let $R=K[x_1,\ldots,x_n]$ be a polynomial ring and let $R'=K[X_1,\ldots,X_n]$ be a divided power algebra (see \cite[Appendix A]{IK} or \cite[Appendix A.2.4]{Eisenbud}), regarded as a $R$-module with the \emph{contraction} action
\[
x_i\circ X_j^k=
\begin{cases}X_j^{k-1}\delta_{ij} \  & \text{if} \ k>0\\ 0 & \text{otherwise,}\\
\end{cases}
\]
where $\delta_{ij}$ is the Kronecker delta. We regard $R$ as a graded $K$-algebra with $\deg X_i=\deg x_i$.

For each degree $i\geq 0$, the action of $R$ on $R'$ defines a non-degenerate $K$-bilinear pairing
\begin{equation}
\label{eq:MDPairing}
    R_i \times R'_i \longrightarrow K \text{ with } (f,F) \longmapsto f \circ F.
\end{equation}
Thus for each $i\geq 0$ we have an isomorphism of $K$-vector spaces $R'_i\cong \Hom_K(R_i,K)$ given by $F\mapsto\left\{f\mapsto f\circ F\right\}$.

 It is well known by work of F.\,S.\,Macaulay \cite{Macaulay} that an Artinian $K$-algebra $A=R/I$ is Gorenstein with socle degree $d$ if and only if  $$I=\Ann_R(F)=\{f\in R\mid f\circ F=0\}$$ for some homogeneous polynomial $F\in R'_d$.  The polynomial $F$ is called  {\em Macaulay dual generator} for $A$; it is unique up to a scalar multiple. To emphasize the relationship between $F$ and the Artinian Gorenstein ring defined by its annihilator, we write $A_F=R/\Ann_R(F)$. 

 We now recall a well-known result on the structure of Artinian Gorenstein rings of codimension three due to Buchsbaum-Eisendbud \cite{BE}.

\begin{theorem}[{\cite[Theorem 2.1]{BE}}]\label{lem:BE}
Let $R$ be a Cohen-Macaulay local or graded ring with maximal ideal $\mathfrak{m}$ and let $I$ be an ideal of $R$ of height three. Then $R/I$ is Gorenstein if and only if there is an odd integer $n\geq 3$ such that $I$ is minimally generated by the $(n-1)$--st order pfaffians of some $n \times n$ skew-symmetric matrix with entries in $\mathfrak{m}$.
\end{theorem}

\subsection{Connected sums}

For any Artinian Gorenstein (or AG)  graded $K$-algebra  $A$ with socle degree $d$  and for any non-zero (surjective) morphism of graded vector spaces $f_A\colon A\to K(-d)$, known as an {\it  orientation} of $A$,
there is a pairing
\begin{equation}\label{pairing}
A_{i}\times A_{d-i}\to K \text{ defined by } (a_i,a_{d-i})\mapsto f_A(a_ia_{d-i})
\end{equation}
 which
is non-degenerate. We call the pair $(A,f_A)$ an
{\em oriented AG $K$-algebra}. 

A choice of orientation on $A$ corresponds to a choice of Macaulay dual generator in the following way. Every orientation on $A$ can be written as the function 
$$
f_A \colon A\to K 
\quad \mbox{defined by}\quad  
f_A(g)\mapsto (g\circ F)(0)
$$ 
for some Macaulay dual generator $F$ of $A$, where the notation $(g\circ F)(0)$ refers to evaluating the element $g\circ F$ of $R'$ at $X_1=\cdots=X_n=0$.

Let $(A,f_A)$ and $(T,f_T)$ be two oriented AG $K$-algebras with a degree-preserving map  $\pi\colon A \to T$ and 
$\reg(A)=d \geq \reg(T)=k$.
By  {\cite[Lemma 2.1]{IMS}},
there exists a unique homogeneous element $\tau_A \in A_{d-k}$, called the {\em Thom class} for $\pi\colon A \to  T$, such that 
$$
f_A(\tau _A a)=
f_T(\pi (a)) \qforall a\in A.
$$
Note
that the  Thom class for $\pi\colon A \to  T$ depends not only on the map $\pi $, but also on the orientations chosen for $A$ and $T$.

\begin{example}
 Let $(A,f_A)$ be an oriented AG $K$-algebra with socle degree $\reg(A)=d$. Consider $(K,f_K)$ where $f_K\colon K\to K$ is the identity map. Then the Thom class for the canonical projection $\pi\colon A\to K$ is the unique element $s\in A_d$ such that $f_A(s)=1$.
\end{example}

For oriented AG $K$-algebras 
$$(A,f_A), (B,f_B), (T,f_T)
\qwith
 \reg(A)=\reg(B)=d 
 \qand 
 \reg(T)=k,
 $$ let
$$
\pi_A \colon A \to  T 
\qand 
\pi_B \colon B \to T
$$ 
be surjective graded $K$-algebra morphisms with Thom classes 
$\tau_A\in A_{d-k}$ and $\tau_B\in B_{d-k}$, respectively.
From this data, one forms the fiber product algebra $A\times_TB$ as the categorical pullback of $\pi_A,\pi_B$ in the category of graded $K$-algebras.
We assume that 
$$
\pi_A(\tau_A) = \pi_B(\tau_B), 
\quad \mbox{so that} \quad 
(\tau_A,\tau_B) \in A \times _T B.
$$
\begin{definition}\label{def: connected sum}
The {\em connected sum} $A \#_TB$ of  $A$ and $B$ over $T$ is the quotient ring of the fiber product
$$
A\times _T B = 
\bigl\{(a, b) \in A\oplus B \mid  \pi_A(a) = \pi _B(b) \bigr\}
$$
by the principal ideal generated by the pair of Thom classes
$(\tau_A, \tau_B)$; in other words,
$$ 
A \#_TB =
(A \times_T B)/
\langle (\tau_A, \tau_B) \rangle.
$$
\end{definition}
By \cite[Lemma 3.7]{IMS}, the connected sum $A \#_TB$ is  characterized by the following exact sequence of vector spaces:
\begin{equation}\label{exactCS}
    0 \to  T(k - d) \to  A \times _T B \to  A\# _T B \to  0.
\end{equation}
Therefore, the Hilbert series of the connected sum satisfies
\begin{equation}\label{HilbertCS}
    H_{A\# _T B}( t) = H_A( t) + H_B( t)- (1 + t^{d-k})H_T( t).
\end{equation}
\subsection{Lefschetz properties}  
\begin{definition}\label{def: WLP}
Let $A=R/I$ be a graded Artinian $K$-algebra. We say that $A$ has the {\em Weak Lefschetz property} (WLP, for short)
if there is a linear form $\ell \in [A]_1$ such that, for all
integers $i\ge0$, the multiplication map
\[
\times \ell\colon [A]_{i}  \longrightarrow  [A]_{i+1}
\]
has maximal rank (i.e., it is either injective or surjective).
 In this case, the linear form $\ell$ is called a {\em Lefschetz
element} of $A$. If for the general form $\ell \in [A]_1$ and for an integer $j$ the
map $\times \ell:[A]_{j-1}  \longrightarrow  [A]_{j}$ does not have maximal rank, we will say that the ideal $I$ fails the WLP in
degree $j$.
\end{definition}
\begin{definition}
$A$ has the {\em Strong Lefschetz property} (SLP, for short) if there is a linear form $\ell \in [A]_1$ such that, for all
integers $i\ge0$ and $k\ge 1$, the multiplication map
\[
\times \ell^k\colon [A]_{i}  \longrightarrow  [A]_{i+k}
\]
has maximal rank.  Such an element $\ell$ is called a {\em Strong Lefschetz element} for $A$.
\end{definition}
To determine whether an Artinian standard graded $K$-algebra $A$ has the WLP/SLP seems a simple problem of linear algebra, but instead, it has proven to be extremely elusive. 
Part of the great interest in the WLP stems from the ubiquity of its presence and there is a long series of papers determining classes of Artinian algebras holding/failing the WLP/SLP, but much more work remains to be done (see \cite{JMR}).

 Due to their duality properties, the WLP for Artinian Gorenstein algebras is determined by a single map.

\begin{lemma}[{\cite[Proposition 2.1]{MMN}}]\label{mid map}
    Let $A$ be an AG algebra with $\reg(A)=d$. Then the following are equivalent 
    \begin{enumerate}[label=(\roman*),itemsep=0.1cm,topsep=0.1cm]
        \item $\ell\in A_1$ is a weak Lefschetz element on $A$ 
        \item the map $\times \ell: A_{\lfloor\frac{d-1}{2}\rfloor} \to A_{\lfloor\frac{d-1}{2}\rfloor+1}$ has maximum rank (is injective) 
        \item the map $\times \ell: A_{\lfloor\frac{d}{2}\rfloor} \to A_{\lfloor\frac{d}{2}\rfloor+1}$ has maximum rank (is surjective).
        \end{enumerate}
\end{lemma}

The next lemma allows us to relate the Lefschetz property for a graded algebra to the Lefschetz property of the quotient algebra defined by its initial ideal.

\begin{lemma}[{\cite[Proposition 2.9.]{Wiebe}}]\label{lem:Wiebe}
    Let $I\subset R$ be a graded ideal primary to the homogeneous maximal ideal, and let $J$ be the initial ideal of $I$ with respect to some term order. 
    If $R/J$ has WLP (respectively, SLP), then the same holds for $R/I$.
\end{lemma}

In view of Lemma \ref{lem:Wiebe}, to determine whether $R/I$ has  the SLP or WLP, one can assume that $I$ is a monomial ideal. 
In \cite{Chase,CookNagel}, the authors have described a family of monomial ideals that have the SLP or WLP. 
We recall the following result, which will be used later.
\begin{theorem}[{\cite[Theorem 4.1]{Chase}}]\label{thm: Chase}
Let $I=\langle x^a,y^b,z^c,x^\alpha z^\gamma,y^\beta z^\gamma\rangle$ be an Artinian monomial ideal in $S=K[x,y,z]$, where $0<\alpha<a,$ and $0<\gamma <c$. 
Then $S/I$ has the Strong Lefschetz property if any of the following conditions hold:
\begin{enumerate}[itemsep=0.15cm,label=(\arabic*)]
    \item $\alpha+\beta-1 \leq a+b-c \leq \alpha+\beta+1;$
    \item $\min\{\alpha,\beta\} \neq \max\{\alpha,\beta\}=\min\{\alpha+\beta,a,b\}$, and \\
    $\max\{\alpha,\beta\}-\gamma-1 \leq a+b-c \leq \max\{\alpha,\beta\}-\gamma+1;$
    \item $\min\{\alpha,\beta\} < \max\{\alpha,\beta\} \leq 2$, and $a+b+\gamma \leq c+2.$
\end{enumerate}
    
\end{theorem}

\section{Artinian Gorenstein algebras with binomial Macaulay dual generator}

 We first establish that  when the residue field contains roots of unity, we may assume that the dual generators do not involve any coefficients other than $\pm 1$.
 
\begin{lemma}\label{lem: no constants}
Assume that $K$ is algebraically closed (or at least that it contains all roots of unity) and $R'=K[X_1,\ldots,X_n]$. Then for any binomial $G\in R'$ there is a binomial $F=m_1-m_2\in R'$ so that $m_1,m_2$ are monomials and $A_G\cong A_F$.
\end{lemma}
\begin{proof}
Since the Macaulay dual generator is only unique up to scalar multiple, we may assume $G=m_1+\theta m_2$ where $m_1,m_2$ are monomials and $\theta\in K$. Pick a variable, say $X_i$, so that the exponent $u$ of $X_i$ in $m_1$ is different from the exponent $v$ of $X_i$ in $m_2$. Set $\lambda$ to be a $(v-u)$th root of $-\theta^{-1}$ and observe that the change of coordinates mapping $X_i\mapsto \lambda X_i$ takes $G$ to $\lambda^u(m_1-m_2)$. Thus this change of coordinates provides an isomorphism $A_G\cong A_F$.
\end{proof}

In view of \cref{lem: no constants} henceforth when we say $F$ is a binomial we shall always mean that for a pair of monomials  $m_1, m_2 \in R '= K[X_1,\dots ,X_n]$, $$F=m_1-m_2$$ 
and, if necessary, we will assume that $K$ is algebraically closed.

When  $m_1$ and $m_2$ are expressed in disjoint sets of variables, as is the case in \cref{l:gcd1}, then \cite[Proposition 4.19]{IMS} shows that $A_{m_1-m_2}=A_{m_1}\#_K A_{m_2}$.
Moreover, it is shown in \cite[Proposition 5.7]{IMS} that if $A$ and $B$ satisfy the SLP and  have the same socle degree, then $A\#_K B$ also satisfies the SLP. We will see that this is no longer true when  the connected sum is not taken over a field (see, for example, \cref{e:failure}). 
The following statement follows.

\begin{lemma}\cite[Proposition 3.77]{LefscetzBook}\label{l:gcd1}
    Let $A_F$ be an Artinian Gorenstein algebra of codimension $c\ge 3$ with  dual generator a binomial $F=m_1-m_2$. If $\gcd(m_1,m_2)=1$, then $A_F$ satisfies the SLP.
\end{lemma}

Our next goal is to generalize \cref{l:gcd1} and give large families of binomials $F$ whose associated algebra $A_F$ satisfies WLP. Our main tool will again be the decomposition of $A_F$, under certain circumstances, as connected sums of complete intersection monomial  rings. To this end, we state the following useful decomposition criterion,  which also allows us to give a formula for the Hilbert function of $A_F$ whenever it is decomposable as a connected sum.

\begin{proposition}\label{key} 
   Let $g,m_1, m_2$ be monomials in $R'$. Let $A_F$ be an Artinian Gorenstein algebra  with Macaulay dual generator a homogeneous binomial $$F=g(m_1-m_2) \qwhere  \gcd(m_1,m_2)=1, \quad m_1\nmid g \qand m_2\nmid g.$$ Suppose that 
    $$
        g\cdot m_1 = X_1^{a_1}\cdots X_n^{a_n}, \quad 
        g\cdot m_2 = X_1^{b_1}\cdots X_n^{b_n} \qand
        g = X_1^{c_1}\cdots X_n^{c_n}.
    $$
    Then
    \begin{enumerate}
        \item\label{statement1} $A_F=A_{gm_1} \#  _{A_g} A_{gm_2}$,
        \item\label{statement2}  the Hilbert series of $A_F$ is
    \[
    H_{A_F}(t)=\frac{\prod_{i=1}^n(1-t^{a_i+1})+\prod_{i=1}^n(1-t^{b_i+1})-(1+t^{d-k})\prod_{i=1}^n(1-t^{c_i+1})}{(1-t)^n}.
    \]
    \end{enumerate}
\end{proposition}
\begin{proof} Statement \eqref{statement1} was previously shown in \cite[Proposition 4.19]{IMS}. 

\eqref{statement2} The formula for the Hilbert series follows from \eqref{HilbertCS}.
\end{proof}

Our first main result is based on the principle that a connected sum algebra, where the summands have the WLP, also has the WLP, provided that the socle degree of the algebra over which the connected sum is taken is not too large. Indeed, we have:

\begin{theorem} \label{main1}
Assume $K$ is a field of characteristic zero. Let $g,m_1, m_2$ be monomials in $R'$.    Let $A_F$ be an Artinian Gorenstein algebra of socle degree $d$, with Macaulay dual generator $$F=g(m_1-m_2), \qwhere \gcd(m_1,m_2)=1 \qand \deg (g)<\left \lfloor\frac{\deg (F)-1}{2} \right \rfloor.$$ Then, $A_F$ has the WLP, and hence the Hilbert function of $A_F$ is unimodal.
\end{theorem} 
\begin{proof}
 The hypothesis $\deg (g)<\left \lfloor\frac{\deg (F)-1}{2} \right\rfloor$ implies that $m_1\nmid g$ and $m_2\nmid g$. Therefore, we can apply \cref{key} to obtain $$A_F=A_{gm_1} \#  _{A_g} A_{gm_2}.$$
Using  the degree-preserving  exact sequences in \eqref{exactCS}:
$$ 0 \to  A_{g}(\deg(g) - d) \to  A_{gm_1}\times _{ A_{g} } A_{gm_2}\to  A_{gm_1} \#  _{A_g} A_{gm_2}=A_F\to  0
$$
and
$$
  0\to A_{gm_1} \times  _{A_g} A_{gm_2} \to A_{gm_1}\oplus A_{gm_2} \to {A_g} \to 0\ ,
$$
we get a direct sum decomposition \begin{equation}\label{eq: direct sum}
[A_F]_t\cong [A_{gm_1}]_t\oplus [A_{gm_2}]_t \qfor  t=\left \lfloor\frac{d-1}{2}\right\rfloor, \left \lfloor\frac{d-1}{2}\right\rfloor+1.
\end{equation}
This conclusion follows from the above exact sequences and the condition that $[A_g]_t=0$ and $[A_g(\deg(g)-d)]_t=0$ for the values of $t$ given above. The latter condition holds because the socle degree of $A_g$ is given by
\[
\reg(A_g)=\deg(g)<\left\lfloor\frac{d-1}{2}\right\rfloor\leq t,
\]
while the initial degree of $A_g(\deg(g)-d)$ is 
\[
d-\deg(g)>d-\left \lfloor\frac{d-1}{2}\right \rfloor=\left \lfloor\frac{d-1}{2}\right \rfloor+1\geq t.
\]
Since $A_{gm_1}$ and $A_{gm_2}$ are monomial complete intersection of socle degree $d$, they have the SLP in characteristic zero, Consequently we know that for a general linear form $\ell$ the multiplication map:
$$\times \ell \colon [A_{gm_1}\oplus A_{gm_2}]_{\lfloor\frac{d-1}{2}\rfloor}\longrightarrow [A_{gm_1}\oplus A_{gm_2}]_{\lfloor\frac{d-1}{2}\rfloor +1}$$
is injective (as the direct sum of two injective maps) and we conclude from \eqref{eq: direct sum} that  for a general linear form $\ell$
  $$\times \ell :[A_F]_{\lfloor\frac{d-1}{2}\rfloor} \to [A_F ]_{\lfloor\frac{d-1}{2}\rfloor+1} $$
  is injective. Therefore, by \cref{mid map}, $A_F$ has the WLP.
\end{proof}

\begin{remark} 
The reader may consult \cite{ADFMMSV1} for new examples of families of Artinian Gorenstein algebras $A_F$ of codimension $n\ge 4$ with a binomial Macaulay dual generator  as in \cref{main1}, satisfying the WLP and having 
    $\deg (g)\ge \left \lfloor\frac{\deg (F)-1}{2} \right \rfloor$. Thus \cref{main1} provides a sufficient, but not always necessary condition for the WLP. 
\end{remark}

A useful criterion to establish the WLP for AG algebras is expressed in terms of higher Hessian matrices defined as follows.
\begin{definition}
Let $F\in R' = K[X_1, \dots  , X_n]$ be a homogeneous polynomial and let $ A = R/\Ann_R(F)$ be the
associated AG algebra. Fix  an ordered
$K$-basis $$\mathcal{B} = \{w_j \mid 1\le j \le h_t:=\dim_K A_t \} \subset A_t.$$
The $t$-th (relative) {\em Hessian matrix} of $F$ with respect to $\mathcal{B}$ is defined as the $h_t \times h_t$ matrix:
$$
\Hess_F^t=(w_iw_j(F))_{i,j}.$$
\end{definition}
In the following example, we use the Hessian criterion proved in {\cite[Theorem 4]{w1} and \cite[Theorem 3.1]{MW}} to show that, in codimension $c\ge 4$, it is not true that all connected sum algebras have WLP. 
This example also shows that \cref{main1} is  the best possible in the sense that the hypothesis $\deg (g)<\left \lfloor\frac{\deg (F)-1}{2} \right \rfloor$ cannot be removed.  The first example of a codimension four artinian Gorenstein algebra that fails to satisfy the WLP was given by Ikeda in \cite{Ikeda}; see also \cite[Example 3.79]{LefscetzBook}. Ikeda's example has Hilbert function $(1,4,10,10,4,1)$, but \cref{e:failure} has a smaller Hilbert function, which is least possible for such an example in view of \cite{BMMN}.

\begin{example}[{\bf The failure of WLP in codimension $\geq 4$}]\label{e:failure}
    We take $$F=X^3YZ-XY^3T=XY(X^2Z-Y^2T)\in K[X,Y,Z,T],$$ so we have 
    $$\deg(XY)=2=\lfloor 4/2 \rfloor =\left \lfloor \frac{\deg(F)-1}{2} \right \rfloor.$$
      The Hilbert function of $A_F$ is $(1, 4, 7, 7, 4, 1)$,
    $$\Ann(F)=\langle z^2,t^2,tz,x^2t,y^2z,x^2z+y^2t,y^4,x^2y^2,x^4\rangle,$$
 and we will show that for any linear form $\ell \in [A_F]_1$, the multiplication map $$\ell \colon[A_F]_2\longrightarrow [A_F]_3$$ has rank $<7$. The Hessian matrix of $F$ of order two is of the following form
$$
\Hess^2_F=6\begin{pmatrix}
0&Y&X&Z&0&0&0\\
Y&0&0&X&0&0&0\\
X&0&0&0&0&0&0\\
Z&X&0&0&0&-Y&-T\\
0&0&0&0&0&0&-Y\\
0&0&0&-Y&0&0&-X\\
0&0&0&-T&-Y&-X&0\\
\end{pmatrix},
$$
and it has a vanishing determinant. 
According to the Hessian criterion \cite[Theorem 3.1]{MW}, $A_F$ does not have WLP. 
It is worthwhile to point out that $A_F$ is a doubling (\cref{doubling_const}) of the 0-dimensional subscheme $P_F\subset \PP^3$, with homogeneous ideal $J=\langle z^2,t^2,tz,x^2t,y^2z,x^2y^2\rangle$.

 More generally when $$F=X_1X_2X_5\cdots X_n(X_1^{n-2}X_3-X_2^{n-2}X_4),$$ computations on Macaulay2~\cite{M2} indicate that the codimension $n$ artinian algebra $A_F$ fails the WLP for all values  $n\in \{4,\ldots, 10\}$, which is as far as  we have been able to compute. Based on this evidence and a careful analysis of the behaviour of the binomial $F$, we speculate that $A_F$ fails the Weak Lefschetz property for all values $n \geq 4$.
\end{example}

\section{Structure of codimension 3 algebras with binomial Macaulay dual generator}

From now on, we restrict our attention to the codimension $c=3$ case. This case is particularly appealing because of the Buchsbaum-Eisenbud structure theorem \cref{lem:BE} describing the Artinian Gorenstein algebras of dimension three. Despite this powerful theorem, from the Macaulay duality perspective the codimension three case is poorly understood and given a homogeneous polynomial $F\in R'$, an explicit formula for the minimal generators of $\Ann(F)$ is not known, even in the case of $c=3$. We give a complete solution to this problem for $F$ a binomial in \cref{prop: gens family4}.

Up to a permutation of the variables, if $\kk$ is algebraically closed, by \cref{lem: no constants} one can write any binomial Macaulay dual generator in three variables as follows: 
\begin{equation}\label{eq: 3 var binomial}
F=X^aY^bZ^c(Z^n-X^eY^m) \text{ with } n=e+m>0.
\end{equation}
We begin by determining the minimal generators and the minimal free resolutions of the annihilator ideals for the binomials given in \eqref{eq: 3 var binomial}.

\begin{theorem}\label{prop: gens family4}
Consider  the binomial $F=X^aY^bZ^c(Z^n-X^eY^m)$ of degree $d$ with $n=e+m>0$. Set $c+1=nq+r$ with $0\leq r<n$ and set
\[
P=\sum_{i=0}^q x^{ei}y^{mi}z^{c+1-ni}.
\]
Then 
\begin{enumerate}
 \item[(0)]\label{Prop4.2-case0} if $e=0$ or $m=0$ then  
 $A_F$ is a complete intersection.
 
 \end{enumerate}
 For all the following cases we assume $em\neq 0$. Then \[\Ann(F)= \langle  x^{a+e+1}, y^{b+m+1}\rangle + J,\] where
 \begin{enumerate}
 \item\label{Prop4.2-case1} if  $a<qe$ or $b<qm$ then $J=\langle P\rangle$ and $A_F$ is a complete intersection;

\item\label{Prop4.2-case2} if $a\geq (q+1)e$ and $b\geq (q+1)m$ and $r>0$ then $J$ is minimally generated by
\begin{multline*}
z^{n+c+1},\, x^{a-(q+1)e+1}\left[z^{n-r}P+x^{(q+1)e}y^{(q+1)m}\right],\\
 y^{b-(q+1)m+1}\left [z^{n-r}P+x^{(q+1)e}y^{(q+1)m}\right],\, x^{a-qe+1}P,\, y^{b-qm+1}P;
\end{multline*}
\item \label{Prop4.2-case4} if $a\geq (q+1)e$ and $b\geq (q+1)m$ and $r=0$ then $J=\big \langle x^{a-qe+1}P, y^{b-qm+1}P ,z^{n+c+1}\big \rangle$;

\item\label{Prop4.2-case3} otherwise 
$J= \big \langle x^{a-qe+1}P, y^{b-qm+1}P , z^{n-r}P+x^{(q+1)e}y^{(q+1)m}\big \rangle$.
\end{enumerate}
Moreover, in each case, the ideal $\Ann(F)$ is minimally generated by the given generators
and the minimal free resolutions of $R/\Ann(F)$ appear in \cref{f:res}. 

\begin{table}[ht!]
\begin{tabular}{|p{0.4cm}|p{13.45cm}|}
\hline
\eqref{Prop4.2-case1} &
\hspace{1cm}$0 \rightarrow R(-d-3) \rightarrow 
 \begin{array}{c} R(-a-b-e-m-2)\\\oplus \\ R(-a-c-e-2) \\\oplus  \\ R(-b-c-m-2) \end{array}
    \rightarrow \begin{array}{c}  R(-a-e-1) \\ \oplus \\ R(-b-m-1) \\ \oplus \\ R(-c-1) \end{array} \rightarrow R$
  \\
  \hline
 \eqref{Prop4.2-case2} & 
 \parbox{0.7\linewidth}{ \begin{multline}\label{eq:free-res-case-2}
 \textstyle{\hspace{-0.3cm}0 \rightarrow R(-d-3) \rightarrow 
 \begin{array}{c} 
 R(-b-c-m-2)\\\oplus \\ R(-a-c-e-2) \\\oplus \\ R(-a-b-2)\\ \oplus \\ R(-a-(q+1)m-r-1)\\\oplus \\ R(-b-(q+1)e-r-1) \\ \oplus \\ R(-b-(q+1)e-m-1)  \\ \oplus \\ R(-a-(q+1)m-e-1) 
 \end{array}
\rightarrow
\begin{array}{c} 
R(-a-e-1)\\\oplus \\ R(-b-m-1)\\\oplus \\ R(-n-c-1)\\\oplus \\R(-b-(q+1)e-1) \\ \oplus \\  R(-a-(q+1)m-1) \\ \oplus \\ R(-a-qm-r-1) \\ \oplus \\ R(-b-qe-r-1) \end{array} \rightarrow R}
  \end{multline}} 
  \\
  \hline
\eqref{Prop4.2-case4} & 
\parbox{0.7\linewidth}{ 
\begin{multline}\label{eq: free res case 3}
\textstyle{\quad 0 \rightarrow R(-d-3) \rightarrow \begin{array}{c} R(-b-c-m-2)\\\oplus \\ R(-a-c-e-2) \\ \oplus\\ R(-b-(q+1)e-m-1)  \\ \oplus \\ R(-a-(q+1)m-e-1)\\ \oplus \\ R(-a-b-2) \end{array}
    \rightarrow \begin{array}{c} R(-a-e-1)\\\oplus \\ R(-b-m-1)\\ \oplus \\ R(-a-qm-1) \\ \oplus \\ R(-b-qe-1) \\ \oplus \\ R(-c-n-1) \end{array} \rightarrow R \hspace{0.2cm}}
\end{multline}}
  \\
  \hline
\eqref{Prop4.2-case3} &
 \parbox{0.7\linewidth}{ \begin{multline}\label{eq: free res case 4}
 \textstyle{0 \rightarrow R(-d-3) \rightarrow \begin{array}{c} R(-b-c-m-2)\\\oplus \\ R(-a-c-e-2) \\ \oplus\\ R(-b-(q+1)e-m-1)  \\ \oplus \\ R(-a-(q+1)m-e-1)\\ \oplus \\ R(-a-b-r-2) \end{array}
    \rightarrow \begin{array}{c} R(-a-e-1)\\\oplus \\ R(-b-m-1)\\ \oplus \\ R(-a-qm-r-1) \\ \oplus \\ R(-b-qe-r-1) \\ \oplus \\ R(-c-n+r-1) \end{array} \rightarrow R  }
    \end{multline}}\\
  \hline
 \end{tabular}
\caption{Resolutions for \cref{prop: gens family4}.}\label{f:res}
\end{table}

\end{theorem}
\begin{proof}
We start with case (0) where $F=X^aY^bZ^c(Z^n-Y^n)$ or $F=X^aY^bZ^c(Z^n-X^n)$. By symmetry, it suffices to handle the polynomial $F=X^aY^bZ^c(Z^n-Y^n)$. Set $G_1=X^a$ and $G_2=Y^bZ^c(Z^n-Y^n)$. Then from $F=G_1G_2$ and the fact that $G_1$ and $G_2$ do not share variables, it follows that $A_F=A_{G_1}\otimes_K A_{G_2}$. Since all AG algebras of codimension at most two are complete intersections, both $A_{G_1}$ and $A_{G_2}$ are complete intersections. Thus, $A_F=A_{G_1}\otimes_K A_{G_2}$ is a complete intersection as well with 
\[
\Ann(F)=\Ann(G_1)+\Ann(G_2)=\langle x^{a+1} \rangle +\Ann(G_2).
\]
The common structure of the proof in the remaining cases is as follows. 
In each case, we identify an ideal $I$ such that $I\subseteq \Ann(F)$. 
We then use the Buchsbaum-Eisenbud criterion \cite{BE} to show that $R/I$ is a Gorenstein algebra of codimension three by constructing a skew-symmetric matrix whose pfaffians generate $I$. This matrix also serves to establish in each case that $\reg(R/I)=d$. Finally, the containment $I\subset \Ann(F)$ determines a degree-preserving surjection $R/I\twoheadrightarrow A_F$. Since this map induces a bijection between the socles of $A_F$ and $R/I$, it must be an isomorphism cf.~\cite[Lemma 3.2]{ADFMMSV1}, hence $I=\Ann(F)$.

Let $I= \langle x^{a+e+1}, y^{b+m+1}\rangle + J$.
We note that for $A,B,C\geq 0$\begin{equation}\label{e:pof0}
 x^Ay^Bz^C \circ F = 
 X^{a-A}Y^{b-B}Z^{c+n-C} -
 X^{a+e-A}Y^{b+m-B}Z^{c-C},
\end{equation}
where we adopt the convention that any monomial containing a negative exponent is equal to zero.
It is clear from \eqref{e:pof0} that $x^{a+e+1}, y^{b+m+1} \in \Ann(F)$. 

Now consider $P \circ F $, we have  that 
\begin{align}\label{e:pof}
 P \circ F &= 
 \left [\sum_{i=0}^q x^{ie}y^{im}z^{c+1-ni}\right ]  \circ F \\ \nonumber 
  &=\sum_{i=0}^{q} X^{a-ie}Y^{b-im}Z^{n(i+1)-1}  
   -\sum_{i=0}^{q} X^{a-(i-1)e}Y^{b-(i-1)m}Z^{ni-1} \\ \nonumber
  &=\sum_{i=0}^{q} X^{a-ie}Y^{b-im}Z^{n(i+1)-1}  
   -\sum_{i=-1}^{q-1} X^{a-ie}Y^{b-im}Z^{n(i+1)-1} \\ \nonumber
  &= X^{a-qe}Y^{b-qm}Z^{n(q+1)-1} 
   - X^{a+e}Y^{b+m}Z^{-1}\\ \nonumber
    &= X^{a-qe}Y^{b-qm}Z^{c+n-r}.
\end{align}
It follows more generally from \eqref{e:pof} that for $A,B,C\geq 0$
\begin{equation}\label{e:pof2}
 x^Ay^Bz^C P \circ F = X^{a-qe-A}Y^{b-qm-B}Z^{c+n-r-C}.
 \end{equation}
The proof now proceeds by considering the four cases outlined in the statement.

\subsubsection*{\eqref{Prop4.2-case1}} In this case,  $a<qe$ or $b<qm$. 
Using \eqref{e:pof} and recalling that a monomial containing a negative exponent is equal to zero, we see that 
$$P\circ F= X^{a-qe}Y^{b-qm}Z^{(q+1)n-1}
    =0.$$ 
This implies that $P\in \Ann(F)$, and hence $I\subseteq \Ann(F)$.
Observe that $I$ is a complete intersection since, for example, the leading terms of the generators under any term order for which $z>y>x$ form a regular sequence. Moreover, the socle degree of $R/I$ is
\begin{align*}
\reg(R/I) &= a+e+b+m+\ell n+u-1 \\
&= a+b+(\ell+1)n+c-\ell n \\
&=  a+b+c+n=\deg(F)=\reg(A_F).
\end{align*}
Thus  $\Ann(F)=I$ follows by \cite[Lemma 3.2]{ADFMMSV1} as explained above.

\subsubsection*{\eqref{Prop4.2-case2}} In this case 
$a-(q+1)e+1\geq 0$ and $b-(q+1)m+1 \geq 0$.
We check that $J\subset \Ann(F)$.  Clearly, by \eqref{e:pof0}, $z^{n+c+1} \in \Ann(F)$. Next, using \eqref{e:pof2} and \eqref{e:pof0}, we see that the second generator of $J$ is in $\Ann(F)$
because 
\begin{align*}
x^{a-(q+1)e+1}z^{n-r}P\circ F &=  
X^{a-qe-(a-(q+1)e+1)}Y^{b-qm}Z^{c+n-r-(n-r)} =
X^{e-1}Y^{b-mq}Z^{c}\\
&=-x^{a+1}y^{(q+1)m}\circ F,
\end{align*}
and similarly, the third generator of $J$ is in $\Ann(F)$
because
$$y^{b-(q+1)m+1}z^{n-r}P\circ F =  
X^{a-qe}Y^{m-1}Z^{c}=
-x^{(q+1)e}y^{b+1}\circ F.
$$
 Finally, working under the convention that terms with negative exponents are understood to be zero, again from \eqref{e:pof2} we have
\begin{flalign*}
x^{a-qe+1}P\circ F&=
X^{-1}Y^{b-qm}Z^{c+n-r} =0,\\
y^{b-qm+1}P\circ F&=
X^{a-qe}Y^{-1}Z^{c+n-r} =0,
\end{flalign*}
which shows that the last two generators of $J$ belong to $\Ann(F)$.  
This establishes $I\subseteq \Ann(F)$.

It can be verified by direct computation that $I$ is the ideal of Pfaffians of the skew-symmetric matrix

\begin{equation}\label{eq: matrix 2}
\begin{bmatrix}
    0 & 0 & 0 & 0 & -z^r & -x^e & -y^{b-(q+1)m+1}\\
    0 & 0 & 0 & -P & 0 & y^{(q+1)m} & 0\\
    0 & 0 & 0 & x^{(q+1)e} & y^m & z^{n-r} & 0\\
    0 & P & -x^{(q+1)e} & 0 & 0 & 0 & 0\\
    z^r & 0 & -y^m & 0 & 0 & 0 & x^{a-(q+1)e+1}\\
    x^e & -y^{(q+1)m} & -z^{n-r} & 0 & 0 & 0 & 0 \\
    y^{b-(q+1)m+1} & 0 & 0 & 0 & -x^{a-(q+1)e+1} & 0 & 0
\end{bmatrix}. 
\end{equation}
Hence, by \cref{lem:BE} (see also~\cite[Theorem~3.4.1]{BH93}), the quotient $R/I$ is a Gorenstein ring with free resolution given by \eqref{eq:free-res-case-2}. Note that the application of \cref{lem:BE} fails if $e=0$ or $m=0$ or $r=0$ since in that case the entry $z^r$ does not belong to the homogeneous maximal ideal.

\medskip
\medskip

It can be read from the minimal free resolution \eqref{eq:free-res-case-2}  (see~\cite{Geramita}) that both $R/I$ and $A_F$ have socle degree $d$. 
By the considerations presented at the beginning of the proof, it follows that $\Ann(F)=I$.

\subsubsection*{\eqref{Prop4.2-case4}} This case differs from case \eqref{Prop4.2-case2} only by the assumption that $r=0$. We have already checked in case \eqref{Prop4.2-case2} that the generators of $J$ belong to $\Ann(F)$ independent of the value of $R$. Thus we have $I\subseteq\Ann(F)$.

Set $Q=z^n-x^ey^m$. 
It can be verified by direct computation that $I$ is the ideal of Pfaffians of the skew-symmetric matrix 

\begin{equation}\label{eq: matrix 3}
\begin{bmatrix}
    0 & 0 & 0 & 0 &  y^{(q+1)m} & -P\\
    0 & 0 & y^{b-qm+1} & -x^{a-qe+1} & 0\\
    0 & -y^{b-qm+1} &0 & Q & x^{(q+1)e}\\
    - y^{(q+1)m} & x^{a-qe+1} & Q & 0 & 0\\
    P & 0 & -x^{(q+1)e} & 0 & 0
\end{bmatrix}.
\end{equation}
Hence, by \cref{lem:BE}, $R/I$ is a Gorenstein ring with free resolution given by \eqref{eq: free res case 3}.
It can be read from the minimal free resolution \eqref{eq: free res case 3}  that both $R/I$ and $A_F$ have socle degree $d$. 
By the considerations presented at the beginning of the proof, it follows that $\Ann(F)=I$.

\subsubsection*{\eqref{Prop4.2-case3}} Let 
$$    J=\left \langle  x^{a-qe+1}P, y^{b-qm+1}P , z^{n-r}P+x^{(q+1)e}y^{(q+1)m}\right \rangle.
$$
Our earlier computations show that the first two generators of $J$ belong to $\Ann(F)$. 
The last generator also belongs to $\Ann(F)$, as shown by the following identity implied by \eqref{e:pof2} and \eqref{e:pof0},
$$
z^{n-r}P \circ F = 
X^{a-qe}Y^{b-qm}Z^{c} =
-x^{(q+1)e}y^{(q+1)m} \circ F.
$$
Observe that $I$ is the ideal generated by the Pfaffians of
\begin{equation}\label{eq: matrix 4}
\begin{bmatrix}
     0 & 0 & -x^{(q+1)e} & 0 & P\\
     0 & 0 & -z^{n-r} & x^{a-qe+1} &-y^{(q+1)m}  \\
     x^{(q+1)e}& z^{n-r} & 0 & -y^{b-qm+1} &0 \\
     0 &-x^{a-qe+1} & y^{b-qm+1} & 0 & 0\\
     -P & y^{(q+1)m} & 0 & 0 &0 \\
\end{bmatrix}.
\end{equation}
By \cref{lem:BE}, this yields that $R/I$ is Gorenstein and has the free resolution \eqref{eq: free res case 4}. 
This also verifies that $\reg(R/I)=d=\reg(A_F)$, and by the considerations at the beginning of the proof, it follows that $\Ann(F)=I$.
\end{proof}

Our theorem allows us to make a contribution to the following open problem (see \cite{HWW} and its references list):

\begin{problem}
    Characterize the homogeneous forms $F\in K[X,Y,Z]$ whose annihilator $\Ann(F)$ is a complete intersection ideal.
\end{problem}

 Very few answers for this problem are known. For
example, if $F$ is a monomial then $\Ann(F)$ is a monomial complete intersection and in \cite{HWW}
the case of quadratic complete intersections is considered.
As a direct application of \cref{prop: gens family4} we have

\begin{corollary}\label{c:ci}
Consider the binomial 
$$
F=X^aY^bZ^c(Z^n-X^eY^m) \qwith n=e+m>0.
$$
 Set $q=\left \lfloor \frac{c+1}{n}\right\rfloor$. Then $A_F$ is a complete intersection if and only if $e=0$ or $m=0$ or $a < qe$ or $b < qm$.
\end{corollary}
We will now proceed to show that the minimal sets of generators for $\Ann(F)$ given in \cref{prop: gens family4} have a very special property: under a suitable monomial order they form a Gr\"obner basis for the ideal they generate. 

\begin{proposition}\label{prop: GB}
  Each of the sets of generators listed in \cref{prop: gens family4} is a Gr\"obner basis with respect to any monomial order that satisfies $z>x$ and $z>y$. In particular, the  initial ideals with respect to such an order for the ideals enumerated in \cref{prop: gens family4} (the numbering here matches that in \cref{prop: gens family4})  are  
  \begin{enumerate}[itemsep=0.1cm,topsep=0.1cm]
  \item 
    $\bigl\langle  x^{a+e+1},y^{b+m+1},z^{c+n+1} \big\rangle$;
    \item         
    $\bigl\langle  x^{a+e+1},y^{b+m+1},z^{c+n+1},x^{a-(q+1)e+1}z^{(q+1)n}, y^{b-(q+1)m+1}z^{(q+1)n},x^{a-qe+1}z^{c+1},$ \\$ y^{b-qm+1}z^{c+1}\bigr\rangle $;  
       \item $\bigl\langle  x^{a+e+1},y^{b+m+1},z^{n+c+1},x^{a-qe+1}z^{c+1}, y^{b-qm+1}z^{c+1}\bigr\rangle $;
    \item $\bigl\langle  x^{a+e+1},y^{b+m+1},z^{n(q+1)},x^{a-qe+1}z^{c+1}, y^{b-qm+1}z^{c+1}\bigr\rangle $.
\end{enumerate}
In particular, the last two ideals coincide.
\end{proposition}
\begin{proof}
 We denote by $\lt(f)$ the leading monomial of $f$ and by
\[
S(f,g)=\frac{\lcm(\lt(f),\lt(g))}
{\lt(f)}f-\frac{\lcm(\lt(f),\lt(g))}
{\lt(g)}g
\]
the S-polynomial of a pair of polynomials $f$ and $g$. In each case we apply Buchberger's criterion by computing the non-trivial S-polynomials and showing that they reduce to 0 modulo the given generators.

\eqref{Prop4.2-case1} This case is clear as the three polynomials have pairwise coprime leading terms.

 \eqref{Prop4.2-case2} Let
    \begin{align*}
        Q &= x^{a-(q+1)e+1}\left[z^{n-r}P+x^{(q+1)e}y^{(q+1)m}\right ] \text{ with \;} \lt(Q)=x^{a-(q+1)e+1}z^{(q+1)n},\\
        Q'&= y^{b-(q+1)m+1}\left[z^{n-r}P+x^{(q+1)e}y^{(q+1)m}\right ]\text{ with \;} \lt(Q')=y^{b-(q+1)m+1}z^{(q+1)n}.
    \end{align*}
       The remaining S-polynomials are either zero or they reduce to zero by symmetry based on the computations below.

\begin{itemize}[itemsep=0.1cm,topsep=0.1cm]
\item $S(x^{a+e+1}, Q) = x^{(q+2)e}Q-x^{a+e+1}z^{(q+1)n} \equiv 0 \pmod{x^{a+e+1}}$

\item 
 $S(y^{b+m+1}, Q') = y^{(q+2)m}Q'-y^{b+m+1}z^{(q+1)n}  \equiv 0 \pmod{y^{b+m+1}} 
$ 

\item $S(Q,Q')= y^{b-(q+1)m+1}Q-x^{a-(q+1)e+1}Q' =0$ 

\item $S(Q,x^{a-qe+1}P) = x^eQ-x^{a-qe+1}z^{n-r}P=x^{a+e+1}y^{(q+1)m}
\equiv 0 \pmod{x^{a+e+1}}$

\item $S(Q,y^{b-qm+1}P) = y^{b-qm+1}Q-x^{a-qe+1}z^{n-r}y^{b-qm+1}P
= x^{(q+1)e}y^{b+m+1}\equiv 0 \pmod{y^{b+m+1}}$

\item $S(x^{a+e+1},x^{a-qe+1}P) = x^{a+e+1}z^{c+1}-x^{a+e+1}P\equiv 0 \pmod{x^{a+e+1}}$ 

\item $S(y^{b+m+1}, x^{a-qe+1}P)= x^{a-qe+1}z^{c+1}y^{b+m+1}-y^{b+m+1} x^{a-qe+1}P\equiv 0 \pmod{y^{b+m+1}}$ 

\item $S(x^{a-qe+1}P, y^{a-qm+1}P)=  0$

\item  $\alignLongunderstack{
 S(z^{n+c+1}, Q)&= z^rQ-x^{a-(q+1)e+1}z^{n+c+1}\\
 &= x^{a-qe+1}z^{n-r}\sum_{i=1}^qx^{ei}y^{mi}z^{c+1-ni}+x^{a+1}y^{(q+1)m}z^r\\
 &= x^{a-qe+1}y^m\left[P-x^{qe}y^{qm}z^r\right]+x^{a+1}y^{(q+1)m}z^r\\
 &\equiv 0 \pmod{x^{a-qe+1}P}}$\hfil\mbox{}
\item $\alignLongunderstack{
S(z^{c+n+1}, x^{a-qe+1}P)&= x^{a-qe+1}z^{c+n+1}-z^{(q+1)n} x^{a-qe+1}P  \\
&= z^{(q+1)n}x^{a-qe+1}\sum_{i=1}^q x^{ie}y^{im}z^{c+1-in}\\
&= x^{a-(q-1)e+1}y^m\left[P-x^{eq}y^{mq}z^r\right]\\
&= x^{a-(q-1)e+1}y^mP-x^{a+e+1}y^{mq}z^r \\
&\equiv 0 \pmod{(x^{a-(q-1)e+1}P,x^{a+e+1})}.
}$\hfil\mbox{}
\end{itemize} 

\smallskip

\eqref{Prop4.2-case4} is subsumed by \eqref{Prop4.2-case2} as the relevant computations are a subset of those above, ignoring all computations that involve $Q$ or $Q'$.
 
\eqref{Prop4.2-case3} Let $T=z^{n-r}P + x^{(q+1)e}y^{(q+1)m}$ with $\lt(T)=z^{(q+1)n}$. In view of the computations in part (1) and since the S-polynomial of a pair of polynomials that has coprime leading terms always reduces to zero, we  only need to  compute
    \begin{align*}
    S(x^{a-qe+1}P,T) &= z^{n-r}x^{a-qe+1}P-x^{a-qe+1}T=-x^{a+e+1}y^{(q+1)m}\\
    &\equiv 0 \pmod{x^{a+e+1}}\\
    S(y^{b-qm+1}P,T) &= z^{n-r}y^{b-qm+1}P-y^{b-qm+1}T=-x^{(q+1)e}y^{b+m+1}\\
    &\equiv 0 \pmod{y^{b+m+1}}.
    \end{align*}   

    The last two ideals coincide because  in \eqref{Prop4.2-case4} $r=0$ so $c+1=qn$ and $c+n+1=n(q+1)$.
\end{proof}

\section{SLP for codimension 3 algebras with binomial Macaulay dual generator}

It is well known that all AG algebras of codimension three have unimodal Hilbert function; see \cite[Theorem 4.2]{Stanley} or \cite{Zanello}. Given that unimodality is one of the significant consequences of the weak Lefschetz property, it is natural to pose the following question: 

\begin{question}\label{q 4.1}
    Do all codimension 3 Artinian Gorenstein ideals have the WLP? 
\end{question}

The next goal of this section is to answer \cref{q 4.1} for codimension 3 Artinian Gorenstein algebras with a binomial Macaulay dual generator. We do so by reduction to initial ideals leveraging \cref{lem:Wiebe}. To utilize this result one need to first establish the WLP for the quotient algebras  of all the ideals in \cref{prop: GB}. Case (1) is easily dealt with as it is a monomial complete intersection. Cases (3) and (4) have identical initial ideal. The quotient algebras given by these ideals are Artinian and have two-dimensional socle (Cohen-Macaulay type two). For these we shall appeal to the literature, in particular to works by Cook--Nagel \cite{CookNagel} and Chase \cite{Chase} that study the WLP and SLP, respectively, in this situation. Since we need to understand these algebras in detail we begin with a lemma describing some of their properties which shall be subsequently used.  Note that the ideal denoted $U$ in \cref{lem: HF of U} below is the same as both the ideals in \cref{prop: GB} (3) and (4).

\begin{lemma}\label{lem: HF of U}
For any $a,b,c,e,m$ and $q$ such that $a+1>qe$, $b+1>qm$ and $q(e+m)\leq c+1<(q+1)(e+m)$ the ideal
\[
U=\bigl\langle x^{a+e+1},y^{b+m+1},z^{(q+1)(e+m)},x^{a-qe+1}z^{c+1}, y^{b-qm+1}z^{c+1}\bigr\rangle
\]
has the property that $R/U$ is a type two module with symmetric, unimodal Hilbert function and
\[
\reg(R/U)=a+b+c+e+m.
\]
\end{lemma}
\begin{proof}
 Set $e+m=n$ and $c=qn+r$. The given ideal decomposes further as
\[
U=\bigl\langle  x^{a+e+1},y^{b+m+1},z^{c+1} \bigr\rangle \cap \bigl\langle  x^{a-qe+1},y^{b-qm+1},z^{(q+1)n}\bigr\rangle .
\]
Thus the socle of $R/U$ is spanned by the monomials $x^{a+e}y^{b+m}z^c$, of degree $d=a+b+c+e+m$, and $x^{a-qe}y^{b-qm}z^{(q+1)n}$, of degree $d-c$. Since for Artinian rings the regularity is the top socle degree, the claim regarding $\reg(R/U)$ follows.

Consider $L= \bigl\langle x^{a+e+1},y^{b+m+1},z^{c+1} \bigr\rangle $ and $L'=\bigl\langle x^{a-qe+1},y^{b-qm+1},z^{(q+1)n}\bigr\rangle $ so that $L+L'=\bigl\langle x^{a-qe+1},y^{b-qm+1},z^{c+1}\bigr\rangle $. Then
$$H_{R/U}(t)= H_{R/L}(t)+H_{R/L'}(t)-H_{R/(L+L')}(t)$$
which is equal to
\begin{align*}
&\!\begin{multlined}
 =\frac{(1-t^{a+e+1})(1-t^{b+m+1})(1-t^{c+1})}{(1-t)^3}
+  \frac{(1-t^{a-qe+1})(1-t^{b-qm+1})(1-t^{(q+1)n})}{(1-t)^3}\\
-  \frac{(1-t^{a-qe+1})(1-t^{b-qm+1})(1-t^{c+1})}{(1-t)^3}
\end{multlined}\\
&= \frac{(1-t^{a+e+1})(1-t^{b+m+1})(1-t^{c+1})+(1-t^{a-qe+1})(1-t^{b-qm+1})(t^{c+1}-t^{(q+1)n})}{(1-t)^3}\\
&= s(a+e)s(b+m)s(c)+s(a-qe)s(b-qm)s(n-r)t^{c+1},
\end{align*}
where $s(j)=\sum_{i=0}^j t^i$.
Being the Hilbert series of the complete intersection $R/L$, the summand $s(a+e)s(b+m)s(c)$ is a polynomial of degree $d$ having symmetric and unimodal coefficients. Similarly the summand $s(a-qe)s(b-qm)s(n-r)$ is a polynomial of degree $a+b-(q-1)n-r=a+b+n-c=d-2c-1$ having symmetric and unimodal coefficients. Multiplying it by $t^{c+1}$ shifts these coefficients to degrees $c+1$ to $d-c$, thus centering them about the middle degree of the previous polynomial. Finally adding two polynomials with symmetric and unimodal coefficients that are centered about the same degree $d/2$ results in a polynomial with the same property.
\end{proof}

We are now ready for the second main result of this section, which answers \cref{q 4.1} affirmatively for AG algebras with binomial Macaulay dual polynomial. Having outlined most of the steps of the proof before  \cref{lem: HF of U} it remains to remark only that case (2) is the most involved. Here the quotient algebras given by the initial ideal has socle of dimension three. We arrive at the proof in that case by a delicate argument connecting certain type three algebras with appropriate counterparts of type two.

\begin{theorem}\label{thm: family 6}
Assume that $K$ has characteristic zero. Consider the binomial
  \[ F=X^aY^bZ^c(Z^n-X^eY^m) \qwith n=e+m\ge 0. 
  \]
  Then the Artinian Gorenstein algebra $A_F$ satisfies the SLP.
\end{theorem}
\begin{proof}
Recall that $A_F=K[x,y,z]/\Ann(F)$ where $\Ann(F)$ is one of the ideals described in \cref{prop: gens family4}.
We proceed by case analysis based on the three cases outlined in \cref{prop: gens family4}.

In case (0) we reach the desired conclusion by applying \cite[Proposition 3.8]{ADFMMSV1}. 

In the remaining cases, fix any monomial term order having $z>x$ and $z>y$ and let $J$ be the initial ideal of $\Ann(F)$ with respect to this term order. \cref{prop: GB} gives the following possibilities for $J$:

\begin{enumerate}[leftmargin=5em,itemsep=0.1cm]
  \item[\quad \eqref{Prop4.2-case1}]\label{J-case1} $J=\bigl\langle x^{a+e+1},y^{b+m+1},z^{c+1}\bigr\rangle $;
    \item[\quad \eqref{Prop4.2-case2}]  \label{J-case2}  
    $J=\bigl\langle  x^{a+e+1},y^{b+m+1},z^{c+n+1},x^{a-(q+1)e+1}z^{(q+1)n},$
    
    \hfill$ y^{b-(q+1)m+1}z^{(q+1)n},x^{a-qe+1}z^{c+1},$ $ y^{b-qm+1}z^{c+1}\bigr\rangle $;  
    \item[\eqref{Prop4.2-case4} \& \eqref{Prop4.2-case3}] \label{J-case3} $J=\bigl\langle  x^{a+e+1},y^{b+m+1},z^{(q+1)n},x^{a-qe+1}z^{c+1}, y^{b-qm+1}z^{c+1}\bigr\rangle $.
\end{enumerate}

 Let $d=\reg(A_F)$. We show in every case that $R/J$ has the SLP. By \cref{lem:Wiebe}, it then follows the same is true for $A_F$.

 The ideal $J$ in \eqref{Prop4.2-case1} is a monomial complete intersection, which has SLP as shown in \cite{Stanley}.

 The ideal $J$ in \eqref{Prop4.2-case4} and \eqref{Prop4.2-case3} is part of the family of type two monomial ideals discussed in \cite{CookNagel} and \cite{Chase}.  We observe in particular that the criteria in  \cref{thm: Chase} (1) ensure that $R/J$ satisfies SLP since the exponents of the generators of $J$ satisfy
 \[
 (a-qe+1)+(b-qm+1)-1 \leq (a+e+1)+(b+m+1)-(q+1)n \leq (a-qe+1)+(b-qm+1)+1.
 \]
Finally, we consider the ideal $J$ in case \eqref{Prop4.2-case2}. Observe that  it decomposes as $J=U\cap V$, where
\begin{align*}
U&= \bigl\langle  x^{a+e+1},y^{b+m+1},z^{(q+1)n},x^{a-qe+1}z^{c+1}, y^{b-qm+1}z^{c+1}\bigr\rangle ,\\
V&=  \bigl\langle  x^{a-(q+1)e+1},y^{b-(q+1)m+1},z^{c+n+1}\bigr\rangle , \text{ and }\\
U+V&= \bigl\langle  x^{a-(q+1)e+1},y^{b-(q+1)m+1},z^{(q+1)n}\bigr\rangle .
\end{align*}
Note that the ideal $U$ above is the same as the ideal $J$ from case (2).

Let $0\leq i < d/2$. For $\ell\in R_1$ there is a commutative diagram
  \begin{equation}\label{eq: MV}
      \begin{tikzcd}
        0\arrow{r} &{[R/J]}_i\arrow{r}\arrow{d}{\times\ell^{d-2i}}&{[R/U]}_i\oplus {[R/V]}_i \arrow{r} \arrow{d}{\times \ell^{d-2i}} & {[R/(U+V)]}_i\arrow{r}\arrow{d}{\times \ell^{d-2i}} &0\\%
0\arrow{r} &{[R/J]}_{d-i}\arrow{r}&{[R/U]}_{d-i}\oplus {[R/V]}_{d-i} \arrow{r}& {[R/(U+V)]}_{d-i}\arrow{r} &0
      \end{tikzcd}
\end{equation}
in which the map $[R/U]_i\xrightarrow{\times \ell^{d-2i}}[R/U]_{d-i}$ is injective due to $R/U$ having the SLP as shown in case (2) and symmetric unimodal Hilbert function as shown in \cref{lem: HF of U}. Let
\begin{align*}
    L &= \ker\left( [R/V]_i\xrightarrow{\times \ell^{d-2i}}[R/V]_{d-i} \right)\\
    L' &= \ker \left( [R/(U+V)]_i\xrightarrow{\times \ell^{d-2i}}[R/(U+V)]_{d-i}\right).
\end{align*}
We will show that $L$ includes into $L'$ via the map induced by $R/V\mapsto R/(U+V)$. This will establish the claim by ensuring the leftmost vertical map in \eqref{eq: MV} is injective by means of the snake lemma.

Towards this end consider the cyclic module
\begin{equation}\label{eq: U+V/V}
\frac{U+V}{V} \cong z^{(q+1)n}\cdot \frac{K[x,y,z]}{\bigl\langle x^{a-(q+1)e+1},y^{b-(q+1)m+1},z^{r}\bigr\rangle }  =: C,
\end{equation}
which fits into the diagram
 \begin{equation}\label{eq: diagram U+V/V}
      \begin{tikzcd}
        0\arrow{r} &{[(U+V)/V]}_i\arrow{r}\arrow{d}{\times\ell^{d-2i}}&{[R/V]}_i \arrow{r} \arrow{d}{\times \ell^{d-2i}} & {[R/(U+V)]}_i\arrow{r}\arrow{d}{\times \ell^{d-2i}} &0\\%
0\arrow{r} &{[(U+V)/V]}_{d-i}\arrow{r}& {[R/V]}_{d-i} \arrow{r}& {[R/(U+V)]}_{d-i}\arrow{r} &0
      \end{tikzcd}
\end{equation}
The leftmost vertical map in \eqref{eq: diagram U+V/V} corresponds to the map
\begin{equation}\label{eq: C map}
C_{i-(q+1)n}\xrightarrow{\times \ell^{d-2i}} C_{d-i-(q+1)n}
\end{equation}
on the monomial complete intersection $\displaystyle{C=\frac{K[x,y,z]}{\langle x^{a-(q+1)e+1},y^{b-(q+1)m+1},z^r\rangle}}$ which has
\[
\reg(C)=a-(q+1)e+b-(q+1)m+r-1=a+b+r-(q+1)n-1.
\]
Set $j=i-(q+1)n$. Then 
\[
\reg(C)-j=a+b+r-1-i=a+b+c+n-(q+1)n-i=d-(q+1)n-i
\]
and 
\[
d-2i= a+b+c+n -2i=\reg(C)-2j.
\]
Thus the map \eqref{eq: C map} can be rewritten as 
\begin{equation}\label{eq: C map 2}
C_{j}\xrightarrow{\times \ell^{\, \reg(C)-2j}} C_{\reg(C)-j}.
\end{equation}
Note that for $i<(q+1)n$ the map \eqref{eq: C map} is vacuously injective. For $(q+1)n\leq i <d/2$ we have
\[
0\leq j=i-(q+1)n < \frac{d}{2}-(q+1)n=\frac{a+b+r-(q+1)n-1}{2} = \frac{1}{2}\reg(C).
\]
Since $C$ satisfies the  SLP and has symmetric unimodal Hilbert function and since $ j< \reg(C)/2$, we deduce that the map \eqref{eq: C map 2} is  injective, hence the map \eqref{eq: C map} is injective too. By means of the the snake lemma applied to  \eqref{eq: diagram U+V/V}, this establishes injectivity of the induced map $L\to L'$ between the kernels of the last two vertical maps. This yields that the leftmost vertical map in \eqref{eq: MV} is injective and allows to conclude that $R/J$ has SLP. As explained in \cref{lem:Wiebe}, this establishes that $A_F$ has the SLP.
\end{proof}

The AG ring $A_F$ from \cref{e:failure} shares some of the salient features of codimension three AG algebras having binomial Macaulay dual generator. Thus one may wonder what goes wrong in \cref{e:failure}, namely which steps of the proof of \cref{thm: family 6}
cannot be applied in that setting.

\begin{example}\label{counterexample explained}
  Consider the AG ring $A_F$ from \cref{e:failure}.  The set of minimal generators of $I=\Ann(F)$ is a Gr\"obner basis with respect to the reverse lexicographic  order where $z>t>x>y$. The initial ideal of $I$ is an ideal of type three
    \[
\bigl\langle  z^{2},z,t,t^{2},z,x^{2},t,x^{2},x^{4},z,y^{2},x^{2}y^{2},y^{4}\bigr\rangle 
= \bigl\langle x^2,y^4,z,t^2\bigr\rangle\cap \bigl\langle x^4,y^2,z,t\bigr\rangle  \cap \bigl\langle  x^2,y^2,z^2,t \bigr\rangle .
    \]
We consider multiplication maps by a general linear form $\ell$ from degree $\lfloor \frac{d-1}{2}\rfloor =2$ to degree $3$. One could try to imitate the proof of \cref{thm: family 6} taking $U$ to be the intersection of two of the irreducible components above and $V$ to be the remaining component.  Only when $U$ is the intersection of the first two components is the map $[R/U]_2\xrightarrow{\times \ell} [R/U]_3$ injective (in the other cases it is surjective). In this case, however, for the cyclic module
\[
\frac{U+V}{V}=z\cdot \frac{K[x,y,z,t]}{\bigl\langle  x^2,y^2,z,t \bigr\rangle }\]
the map
$[(U+V)/V]_3\xrightarrow{\times \ell} [(U+V)/V]_3$ is not injective (it is surjective).

Thus, the major differences are that \cref{lem: HF of U} does not hold, more specifically $R/U$ need not have a symmetric Hilbert function and need not have the same top degree as $A_F$. Moreover, another important difference is that $\reg((U+V)/V)\neq \reg(A_F)$.
\end{example}

\section{Doubling  in codimension $3$}\label{doubling_const}

As an application of \cref{thm: family 6}   we will prove in \cref{doubling2} that any codimension 3 Artinian Gorenstein algebra with binomial Macaulay dual generator is a doubling of a suitable 0-dimensional subscheme in $\PP^2$.


We begin with the appropriate definitions.
\begin{definition}
Set $R= K[x_1,\ldots,x_n]$. The \emph{canonical module} of a graded $R$-module
$M$ is $\omega_M = \Ext^{n - \dim M}_R (M, \omega _R)$.
\end{definition}

One has $\omega_K \cong K$ and for any homogeneous ideal $I\subset K[x_1,\dots ,x_n]$ of codimension $c$, we have $\omega _{R/I}=\Ext^c_R(R/I,R(-n)).$

\begin{definition}\label{doubling} \cite[Section 2.5]{KKRSSY}
Let $J\subset R$ be a homogeneous ideal of codimension $c$, such that $R/J$ is Cohen-Macaulay
and $\omega_{R/J}$ is its canonical module. Furthermore, assume that $R/J$ satisfies the
condition $G_0$ (i.e., it is Gorenstein at all minimal primes). 
Let $I$ be an ideal of codimension $c + 1$. $I$ is called a {\em doubling } of $J$ via $\psi$ if there exists  a short exact sequence of $R/J$ modules
\begin{equation}\label{eq:doubling}
0 \rightarrow \omega_{R/J}(-d)\stackrel{\psi}{\rightarrow} R/J \rightarrow R/I\rightarrow 0.
\end{equation}
\end{definition}

By \cite[Proposition 3.3.18]{BH93}, if $I$ is a doubling, then $R/I$ is a Gorenstein ring.  
Moreover, the mapping cone of $\psi$ in \eqref{eq:doubling} gives a resolution of $R/I $. 
This mapping cone is the direct sum of the minimal free resolution $F_{\bullet }$ of $R/J$ with its dual (reversed) complex $\Hom(F_{\bullet},R)$ which justifies the terminology of \say{doubling}.

It is not true that every Artinian Gorenstein ideal of codimension $c + 1$ is a doubling of some codimension $c$ ideal (see, for instance,
\cite[Example 2.19]{KKRSSY}).

For our purposes, we first need  a criterion  to determine whether a codimension 3 Artinian Gorenstein algebra $A$ is a doubling of a 0-dimensional subscheme in $\PP^2$ in terms of its Buchsbaum-Eisenbud matrix. To state it we shall fix some additional notation.

Let $I\subset K[x,y,z]$ be an Artinian Gorenstein ideal with a minimal self-dual (up to twist) resolution given by the Buchsbaum-Eisenbud structure theorem (see \cite{BE}):
\begin{equation}
    \label{BE}
    0 \longrightarrow R(-e)\rightarrow \bigoplus _{i=1}^{2n+1} R(-q_i) \xrightarrow{M} \bigoplus _{i=1}^{2n+1} R(-p_i) \xrightarrow{f_1\cdots f_{2n+1}} I \rightarrow 0
\end{equation}
where  \begin{itemize}[itemsep=0.1cm,topsep=0.1cm]
    \item $p_1\le p_2\le \cdots \le p_{2n+1}$;
    \item $q_1\ge q_2\ge \cdots \ge q_{2n+1}$; and
    \item $e-q_i=p_i$ for all $1\le i \le 2n+1$ (by self-duality).
\end{itemize}
Moreover, $M$ is a skew symmetric matrix whose Pfaffians are the $2n+1$ generators $f_1, \cdots ,f_{2n+1}$ of $I$.

\begin{proposition}\label{key_doubling} Utilizing the above notation, assume that the skew-symmetric $(2n+1)\times (2n+1)$ matrix $M$ has the following shape
\begin{equation}\label{eq: skew matrix}
\begin{array}{cc}
& n \hspace{0.3in}  n+1 \\
\begin{array}{l}
 n\\   n+1 \\
\end{array}
&
\left [
\begin{array}{ccc|ccc}

&\mbox{0}&&&\mbox{A}&\\
\hline

&\mbox{$-A^T$}&&&\mbox{B}&\\
\end{array}
\right ]\\
\end{array}
\end{equation}
where $B$ is a $(n+1)\times (n+1)$ skew symmetric matrix. We also assume that the maximal minors of $A$ define a codimension 2 ideal $J\subset R$ with $R/J$ Cohen-Macaulay. Then, $I$ is a doubling of $J$.

Conversely, if $I$ is a doubling of $J$ such that $R/J$ has Hilbert-Burch matrix $A$, then the skew-symmetric matrix of $R/I$ has the form in \eqref{eq: skew matrix}.
\end{proposition}

\begin{proof}
For the direct implication, 
    we first observe that the matrix $A$ is the HIlbert-Burch matrix of  a 0-dimensional subscheme $Z\subset \PP^2$ whose homogeneous ideal $J=I(Z)$ has the following minimal free resolution:
\begin{equation}\label{0dim}
 0 \rightarrow \bigoplus _{n+2}^{2n+1} R(-q_i) \xrightarrow{A} \bigoplus _{i=1}^{n+1} R(-p_i) \xrightarrow{f_1\cdots f_{n+1}} J \rightarrow 0.
\end{equation}
We easily check that the Pfaffians that we get from $M$ deleting the i-th row and column
for $n+1\le i\le 2n+1$ coincide with the maximal minors of $A$. Therefore, we have an inclusion of $J\subset I$.

Dualizing  and twisting by $-e$ the exact sequence \eqref{0dim} and we get a minimal resolution of $\omega_{R/J}(-e+3)$:
$$
 0 \rightarrow R(-e) \rightarrow \bigoplus _{j=1}^{n+1} R(p_j-e) \xrightarrow{A^T} \bigoplus _{i=n+2}^{2n+1} R(-e+q_i) \rightarrow \omega_{R/J}(-e+3) \rightarrow 0.
$$
Since $e-q_i=p_i$ for $1\le i \le 2n+1$, we have:
\begin{equation}\label{eq: res canonical module}
 0 \rightarrow R(-e) \rightarrow\bigoplus _{j=1}^{n+1} R(-q_i) \xrightarrow{A^T} \bigoplus _{i=n+2}^{2n+1} R(-p_i)\rightarrow \omega_{R/J}(-e+3) \rightarrow 0.
\end{equation}
So, $\omega _{R/J}(-e+3)$ has $n$  generators $g_1,\cdots ,g_n$ of degree $p_{n+2}, \cdots ,p_{2n+1}$, respectively, which allow us to construct
a map $$ \psi\colon\omega_{R/J}(-e+3) \to R/J $$
    by mapping the generators of $\omega _{R/J}(-e+3)$ to the elements
$f_{n+2},\cdots ,f_{2n+1}$ of $R/J$, respectively. This map is well defined since these elements of $R/J$ satisfy the same relations that the generators of $\omega_{R/J}$ do, given by the transpose of the matrix $A$, $A^T$, which is equal to $-A$. This yields a short exact sequence
$$
0  \rightarrow \omega_{R/J}(-e+3)  \rightarrow R/J  \rightarrow R/I  \rightarrow 0
$$
which allows us to conclude that $R/I$ is a doubling of $Z$.

The converse follows from the doubling construction explained in \cref{eq:doubling}. Indeed, by hypothesis, we have that $R/J$ and $\omega_{R/J}$ have minimal free resolutions as in \eqref{0dim} and \eqref{eq: res canonical module}, respectively. Taking the mapping cone of $\psi$ in \eqref{eq:doubling} yields the desired conclusion, where the matrix $B$ represents the lift of $\psi$ to a map of complexes restricted to homological degree two.
\end{proof}

We are now able to apply our \cref{prop: gens family4} on the structure of codimension three AG algebras having binomial Macaulay dual generator to conclude that they are doublings.

\begin{corollary}\label{doubling2} Any codimension 3 Artinian Gorenstein algebra with binomial Macaulay dual generator is a doubling of a 0-dimensional subscheme in $\PP^2$.    
\end{corollary}

\begin{proof}
Complete intersections are known to be doublings, settling cases (0) and  \eqref{Prop4.2-case1} of \cref{prop: gens family4}.
It immediately follows from \eqref{eq: matrix 2}, \eqref{eq: matrix 3} and \eqref{eq: matrix 4} in the proof of \cref{prop: gens family4} that the Buchsbaum-Eisenbud matrix of the codimension 3 Artinian Gorenstein algebra $A_F$ is described by a matrix of the form given in  \eqref{eq: skew matrix}. Thus the claim follows by applying \cref{key_doubling}.
\end{proof}

\bigskip 

\bibliographystyle{alpha} %
\bibliography{Journal}
\end{document}